\documentclass[12pt]{article}
\usepackage{amssymb}

\usepackage[utf8]{inputenc}
\usepackage{tikz}
\usetikzlibrary{matrix,arrows,decorations.pathmorphing}
\usetikzlibrary{decorations.pathreplacing} 
\usetikzlibrary{patterns}
\usepackage[hidelinks]{hyperref}
\usepackage[utf8]{inputenc}
\usepackage[english]{babel}
\usepackage{booktabs}
\usepackage{amsmath}
\usepackage{amsthm}
\usepackage{wrapfig}
\usepackage{bm}
\usepackage{mathtools}
\usepackage{caption}
\usepackage{lmodern,wrapfig,amsmath}
\usepackage{enumitem}
\usepackage{ stmaryrd }
\usetikzlibrary{shapes.misc}
\usetikzlibrary{cd}
\usepackage{pgfplots}
\usepackage{tikz-3dplot}
\tdplotsetmaincoords{60}{115}
\pgfplotsset{compat=newest}
\usepackage{ amssymb }
\usepackage{ mathrsfs }
\usepackage[backend=biber, style=ieee]{biblatex}
\bibliography{bibliography}
\usepackage{centernot}
\usepackage{faktor}
\usepackage{amsmath,amsfonts,amssymb, color, braket}

\newcommand{\adj}{\bigcup_{\mathcal{F}}}

\newtheorem{theorem}{Theorem}[section]
\newtheorem{proposition}[theorem]{Proposition}

\newtheorem{definition}[theorem]{Definition}  
\newtheorem{lemma}[theorem]{Lemma}

\title{Non-Hausdorff Manifolds via Adjunction Spaces}
\author{David O'Connell}
\date{\small{\textit{Okinawa Institute of Science and Technology}} \\ \small{\textit{1919-1 Tancha, Onna-son, Okinawa 904-0495, Japan}}}

\begin{document}

\maketitle

\begin{abstract}
    In this paper we will introduce and develop a theory of adjunction spaces which allows the construction of non-Hausdorff topological manifolds from standard Hausdorff ones. This is done by gluing Hausdorff manifolds along homeomorphic open submanifolds whilst leaving the boundaries of these regions unidentified. In the case that these gluing regions have homeomorphic boundaries, it is shown that Hausdorff violation occurs precisely at these boundaries. We then use this adjunction formalism to provide a partial characterisation of the maximal Hausdorff submanifolds that a given non-Hausdorff manifold may admit.
\end{abstract}

\hspace{-5mm}

In this paper we will study non-Hausdorff manifolds. These are locally-Euclidean second-countable spaces that contain points that are ``doubled" or superimposed on top of each other. As the name would suggest, such points cannot be separated by open sets, and thus violate the Hausdorff property. The prototypical example of a non-Hausdorff manifold is the so-called \textit{line with two origins}, pictured below. \\

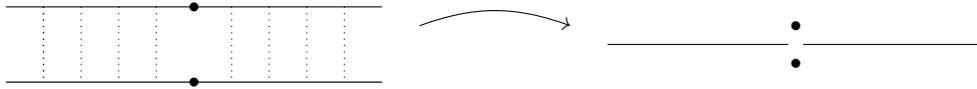
\begin{figure}
    \centering
    \begin{tikzpicture}
 
 \draw[] (-1,0)--(4,0);
 \draw[] (-1,-1)--(4,-1);
 \draw[dotted] (-0.5,0)--(-0.5,-1);
 \draw[dotted] (-0,0)--(-0,-1);
 \draw[dotted] (-0.5,0)--(-0.5,-1);
 \draw[dotted] (0.5,0)--(0.5,-1);
 \draw[dotted] (1,0)--(1,-1);
 \draw[dotted] (2,0)--(2,-1);
 \draw[dotted] (2.5,0)--(2.5,-1);
 \draw[dotted] (3,0)--(3,-1);
 \draw[dotted] (3.5,0)--(3.5,-1);

 \fill(1.5,0) circle[radius=0.06cm] {};
 \fill(1.5,-1) circle[radius=0.06cm] {};

 \draw[->] (4.5,-0.25) to[out=20, in=160] (6.5,-0.25);

 \draw[] (7,-0.5)--(9.4, -0.5);
 \draw[] (9.6, -0.5)--(12,-0.5);
 
 \fill(9.5,-0.25) circle[radius=0.06cm] {};
 \fill(9.5,-0.75) circle[radius=0.06cm] {};
 
\end{tikzpicture}

    \caption{The line with two origins (right), constructed by gluing together two copies of the real line.}
    \label{Fig: Line with two origins}
\end{figure}

The line with two origins can be constructed by gluing together two copies of the real line together everywhere except at the origin. Other examples such as those figures found in \cite{baillif2008manifolds} and \cite{haefliger1957varietes} suggest a general theory for constructing non-Hausdorff manifolds by gluing together Hausdorff ones along open subspaces. A recent result of Placek and Luc confirms that any non-Hausdorff manifold can be built according to such a procedure \cite{luc2020interpreting}. \\

In this paper we will introduce and refine an approach similar to that of \cite{luc2020interpreting} and \cite{OConnellthesis} to further study non-Hausdorff manifolds. We will start by introducing a calculus for adjoining countably-many Hausdorff manifolds together. We will show that any adjunction of countably-many Hausdorff manifolds $M_i$ along open subsets $A_{ij}$ with pairwise homeomorphic boundary components will yield a non-Hausdorff manifold $\textbf{M}$ in which the Hausdorff-violation occurs precisely at the $\textbf{M}$-relative boundaries of the subsets $A_{ij}$. Moreover, in such a situation the manifolds $M_i$ will sit inside $\textbf{M}$ as maximal Hausdorff (open) submanifolds. \\

Interestingly, a non-Hausdorff manifold will often admit infinitely-many maximal Hausdorff submanifolds \cite{muller2013generalized}. This observation motivates the following question: in a non-Hausdorff manifold built by gluing together Hausdorff manifolds $M_i$, are there any topological properties that distinguish the $M_i$ from the other maximal Hausdorff submanifolds? There is a well-known criterion due to Hajicek \cite{hajicek1971causality} that guarantees a given subset is a maximal Hausdorff submanifold. We will spend some time refining this idea, with the eventual conclusion being that if the non-Hausdorff manifold is ``simple" (in a precise sense to be defined later) then the spaces $M_i$ are the unique subspaces satisfying a stricter form of Hajicek's criterion. \\

This paper is organised as follows. In the first section we will introduce a generalised theory of adjunction spaces as colimits of appropriate diagrams. We will identify various conditions which allow certain topological features to be preserved in the adjunction process. With an eye towards the rest of the paper, special attention is paid to those adjunction spaces formed by gluing together topological spaces along open subspaces. \\

In Section 2 we will apply this formalism to the setting of manifolds. We will show that locally-Euclidean second-countable spaces can be formed by gluing together Hausdorff manifolds along homeomorphic open submanifolds. We will then spend some time studying the situation in which the gluing regions have homeomorphic boundary components. We will also argue that non-Hausdorff manifolds built in this way may be paracompact, however they will not admit partitions of unity subordinate to every open cover. We finish Section 2 with a discussion of some known results found regarding Hausdorff submanifolds. \\

In Section 3 we will introduce some examples of non-Hausdorff manifolds built from adjunction spaces. Most of these revolve around Euclidean space, with the exception of a non-Hausdorff sphere \cite{hommelberg2014compact, deeley2022fell}, which we will construct from four copies of punctured spheres. Of particular interest is the branched Euclidean plane, which we will see has infinitely-many maximal Hausdorff submanifolds. Motivated by our examples, in Section 4 we will generalise the result of \cite{muller2013generalized} and show that particularly simple non-Hausdorff manifolds admit only finitely-many maximal Hausdorff submanifolds satisfying a natural condition on their boundaries. \\

Throughout this paper we will assume that all manifolds, Hausdorff or otherwise, are locally-Euclidean, second-countable and connected. We will denote Hausdorff manifolds using standard Latin letters, and we will use boldface characters to emphasise that the manifold in question is potentially non-Hausdorff. All notions of topology used in this paper can be found in standard texts such as \cite{munkrestopology} or \cite{lee2010introduction}. 

\section{Adjunction Spaces}

We start by presenting our formalism for general adjunction spaces. The focus is mainly on the situation in which topological spaces are glued along open sets, since this will be an important precursor to our later discussions of non-Hausdorff manifolds. 

\subsection{Basic Properties}
There are at least two ways to glue together multiple topological spaces in a consistent way. These are: 
\begin{enumerate}
    \item to iterate a binary construction several times over, or 
    \item to glue a collection of spaces together simultaneously.  
\end{enumerate}
The first approach would amount to suitably modifying the standard adjunction spaces found in say \cite{lee2010introduction} or \cite{brown2006topology}. Throughout this paper we will instead focus on the latter case. Formally, gluing together multiple spaces can be achieved by fixing some index set $I$ to enumerate the spaces that we would like to glue together, and by defining a triple of sets $\mathcal{F} := (\textsf{X}, \textsf{A}, \textsf{f})$, where:
\begin{itemize}
    \item the set $\textsf{X}$ is a collection of topological spaces $X_i$, 
    \item the set $\textsf{A}$ is a collection of sets $A_{ij}$ such that $A_{ij} \subseteq X_i$  for all $j \in I$, and 
    \item the set $\textsf{f}$ is a collection of continuous maps $f_{ij}: A_{ij} \rightarrow X_j$.
\end{itemize} 
In order to yield a well-defined adjunction space, we need to impose some consistency conditions on the data contained within $\mathcal{F}$. These conditions are captured in the following definition. 
\begin{definition}\label{adjsystemdef}
A triple $\mathcal{F} = (\textsf{X}, \textsf{A}, \textsf{f})$, is called an adjunction system if it satisfies the following conditions for all $i,j \in I$.
\begin{enumerate}[itemsep=0.7mm]\setlength{\itemindent}{1em}
    \item[\textbf{A1)}] $A_{ii} = X_i$ and $f_{ii} = id_{X_i}$ 
    \item[\textbf{A2)}] $A_{ji} = f_{ij}(A_{ij})$, and $f_{ij}^{-1} = f_{ji}$
    \item[\textbf{A3)}]  $f_{ik}(a) = f_{jk} \circ f_{ij}(a)$ for each $a\in A_{ij}\cap A_{ik}$.
\end{enumerate}
\end{definition}
Observe that the second condition above ensures that each $f_{ij}$ is a homeomorphism. Given an adjunction system $\mathcal{F}$, we can then define the\textit{ adjunction space subordinate to} $\mathcal{F}$, denoted $\bigcup_{\mathcal{F}} X_i$, as the topological space obtained from quotienting the disjoint union $$ \bigsqcup_i X_i := \{ (x,i) \ | \ x \in X_i \} $$ under the relation $\cong$, where $(x,i) \cong (y,j)$ iff $f_{ij}(x) = y$. The conditions of Definition \ref{adjsystemdef} are precisely what is needed to ensure the relation $\cong$ is an equivalence relation. Points in the adjunction space $\bigcup_{\mathcal{F}} X_i$ can be described as equivalence classes of the form $$[x,i] := \{ (y,j) \ | \ f_{ij}(x) = y \}.$$ 
By construction we have a collection of canonical maps $\phi_i: X_i \rightarrow \bigcup_{\mathcal{F}} X_i$ which send each $x$ in $X_i$ to its equivalence class in $\adj X_i$. By construction these maps are continuous and injective. Moreover, these maps will commute on the relevant overlaps, i.e. the equality 
        $$  \phi_j \circ f_{ij} = \phi_i $$
holds for all $i,j$ in $I$. Since the topology of an adjunction space is the quotient of a disjoint union, by construction we have the following useful characterisation of open sets. 
\begin{proposition}\label{adj U open subset iff preimages open in X_i}
A subset $U$ of $\adj X_i$ is open in the adjunction topology iff $\phi_i^{-1}(U)$ is open in $X_i$ for all $i$ in $I$.
\end{proposition}

In the binary version of adjunction spaces, it is well-known that the adjunction of two spaces is the pushout of the diagram below \cite{brown2006topology}. 
\begin{center}
\begin{tikzcd}
Y & A \arrow[l, "f"'] \arrow[r, hook, "\iota_A"] & X
\end{tikzcd}
\end{center}
The following result shows that the adjunction space subordinate to $\mathcal{F}$ can be seen as the colimit of the diagram formed from $\mathcal{F}$.
\begin{lemma}\label{adjuniversalprop}
Let $\psi_i: X_i \rightarrow Y$ be a collection of continuous maps from each $X_i$ to some topological space $Y$, such that for every $i,j \in I$ it is the case that $\psi_i = \psi_j\circ f_{ij}$. Then there is a unique continuous map $g: \adj X_i \rightarrow Y$ and $\psi_i = g \circ \phi_i$ for all $i$ in $I$.
\end{lemma}
\begin{proof}
We define the map $g$ by $g([x,i]) = \psi_i(x)$, that is, $g = \psi_i \circ \phi_i^{-1}$. To see that this defines a function, we need to confirm that $g$ preserves equivalence classes. Suppose that $[x,i] = [y,j]$, i.e. $x = f_{ij}(y)$. Then: $$g([x,i]) = \psi_i(x) = \psi_j(f_{ij}(x)) = \psi_j(y) = g([y,j])$$
as required. We now show that $g$ is continuous. Let $U$ be open in $Y$, and consider the set $g^{-1}(U) = \phi_i \circ \psi_i^{-1}(U)$. Recall the set $g^{-1}(U)$ is open in $\adj X_i$ iff for each $i \in I$, the set $\phi_i^{-1}(g^{-1}(U))$ is open in $X_i$. Observe that: 
$$ \phi_i^{-1}(g^{-1}(U)) = \phi_i^{-1} \circ \phi_i \circ \psi_i^{-1}(U) = \psi_i^{-1}(U)     $$ which is open since $\psi_i$ is continuous. It follows that $g^{-1}(U)$ is open in $\adj X_i$, and thus $g$ is continuous. To see that $g$ is unique, we can use a similar argument to that in \cite{brown2006topology}.
\end{proof}

Throughout the remainder of this paper we will consider adjunction spaces formed by gluing open sets together. In this case, we make the following useful observation. 

\begin{lemma}\label{adjfAopenthen-phi-topoembeds}
Let $\adj X_i$ be an adjunction space formed from $\mathcal{F}$. If each $A_{ij}$ is an open subset of $X_i$, then each $\phi_i$ is an open embedding.  
\end{lemma}
\begin{proof}
Fix some $\phi_i$. By construction $\phi_i$ is injective and continuous, so it suffices to show that $\phi_i$ is an open map. So, let $U$ be an open subset of $X_i$, and consider $\phi_i(U)$. By Prop. \ref{adj U open subset iff preimages open in X_i} this set is open in $\adj X_i$ iff for every $j \in I$, the preimage $\phi_j^{-1} \circ \phi_i(U)$ is open in $X_j$. Observe that $\phi_j^{-1} \circ \phi_i(U) = f_{ij}(U \cap A_{ij})$. Since $U$ is open in $X_i$, the set $U \cap A_{ij}$ is open in $A_{ij}$ (equipped with the subspace topology). By construction $f_{ij}:A_{ij} \rightarrow X_j$ is a embedding. Moreover, each $f_{ij}$ is an open map since we have assumed that each $A_{ij}$ is open. Thus $f_{ij}(U \cap A_{ij})$ is also open in $X_j$, from which it follows that $\phi_j^{-1} \circ\phi_i (U)$ is open in $X_j$ for all $j$. Since $U$ was arbitrary, we may conclude that $\phi_i$ is an open map. The result then follows from the fact that every continuous, open injective map is an open embedding. 
  \end{proof}

\subsection{Subsystems and Subspaces}
Suppose that we have an adjunction space $\adj X_i$ built from the adjunction system $\mathcal{F}$. It follows immediately from Definition \ref{adjsystemdef} that any subset $J$ of the underlying indexing set $I$ will yield another adjunction system $\mathcal{G}$, which can be obtained by simply forgetting the parts of $\mathcal{F}$ that are not contained in $J \subset I$.  As such, there is an associated adjunction space $\bigcup_{\mathcal{G}} X_j$, that consists of gluing only the component spaces $X_j$ where $j$ lies in the subset $J$. In this situation we will refer to $\mathcal{G}$ as an \textit{adjunctive subsystem} of $\mathcal{F}$, and we will write $\mathcal{G} \subseteq \mathcal{F}$. \\

In order to distinguish objects in $\bigcup_{\mathcal{G}} X_j$ from objects in the full adjunction space, we will use the double-bracketed notation $\llbracket \cdot \rrbracket$ to denote objects defined in the adjunctive subspace. In particular, we will use 
$$  \llbracket x,j \rrbracket := \{ (y,k) \ | \ k\in J \ \wedge \ f_{jk}(x) = y \}        $$
to denote the points in $\bigcup_{\mathcal{G}} X_j$. Furthermore, we will denote the canonical maps of the adjunctive subspace by $\chi_j$'s.\\

The universal property \ref{adjuniversalprop} yields a continuous map $g:\bigcup_{\mathcal{G}} X_j \rightarrow \adj X_i$ which will send each equivalence class $\llbracket x,j \rrbracket$ of $\bigcup_{\mathcal{G}} X_j$ into the (possibly larger) equivalence class $[x,j]$ of $\adj X_i$. In general there is no guarantee that the map $g$ acts as a homeomorphism between $\chi_j(X_j)$ and $\phi_j(X_j)$. However, this happens to be the case if we require the gluing regions $A_{ij}$ to be open. 

\begin{theorem}\label{thm adjunctive subspaces}
Let $\mathcal{F}$ be an adjunction system in which each all of the gluing regions $A_{ij}$ are open in their respective spaces. For any adjunctive subsystem $\mathcal{G} \subseteq \mathcal{F}$, the continuous function $g: \bigcup_{\mathcal{G}} X_j \rightarrow \adj X_i$ defined by $\llbracket x,i \rrbracket \mapsto [x,i]$ is an open embedding. 
\end{theorem}
\begin{proof}
By \ref{adjuniversalprop} $g$ is continuous, so it suffices to show that $g$ is both injective and open. For injectivity, let $\llbracket x,j \rrbracket$ and $\llbracket y,k \rrbracket$ be two distinct points in $\bigcup_{\mathcal{G}} X_j$. By construction the equivalence class $\llbracket x,j \rrbracket$ contains all elements of the form $(f_{jl}(x),l)$. Since $\llbracket y,k \rrbracket$ is distinct from $\llbracket x, j\rrbracket$, we may conclude that $f_{jk}(x)$ does not equal $y$, and thus $[x,j]\neq [y,k]$ as well. This confirms that $g$ is injective. 

Suppose now that $U$ is some open subset of the adjunctive subspace $\bigcup_{\mathcal{G}} X_j$. By Prop \ref{adj U open subset iff preimages open in X_i}, $U$ is open in the adjunctive subspace iff all of the preimages $\chi_j^{-1}(U)$ are open in their respective spaces $X_j$. Since we have assumed all $A_{ij}$ are open, Prop. \ref{adjfAopenthen-phi-topoembeds} ensures that all of the canonical maps $\phi_i$ of the adjunction space $\adj X_i$ are open embeddings. Thus the images $\phi_j \circ \chi_j^{-1}(U)$ are open in the spaces $\phi_j(X_j)$. Since these are open subspaces of $\adj X_i$, we may conclude that each $\phi_j \circ \chi_j^{-1}(U)$ is an open subset of $\adj X_i$. It follows that the union $$ \bigcup_{j \in J} \phi_j \circ \chi_j^{-1}(U)  $$
is open in $\adj X_i$. However, this set is precisely equal to the image $g(U)$. Since we chose $U$ arbitrarily, it follows that $g$ is open and therefore $g$ is an open embedding. 
\end{proof}

\subsection{Preservation of Various Properties}
We have seen that the adjunction of arbitrarily-many topological spaces is again a topological space. However, there is no guarantee that the gluing process will preserve any pre-existing structure. The following result is a collection of conditions that suffice to preserve topological features. We state these here without proof, since the arguments involved routinely follow from Lemma \ref{adjfAopenthen-phi-topoembeds} and basic facts of topology. \\

\begin{theorem}\label{preservation of top prop}
Let $\mathcal{F} = (\textsf{X}, \textsf{A}, \textsf{f})$ be an adjunction system with indexing set $I$, and denote by $\textbf{X}$ the adjunction space subordinate to $\mathcal{F}$.
\begin{enumerate}
    \item Suppose that for each $i$ in $I$, the collection $\mathcal{B}_i$ forms a basis for $X_i$. If each $\phi_i$ is an open map, then the collection $\mathcal{B} = \{ \phi_i(B) \ | \ B \in \mathcal{B}_i  \}$ forms a basis for $\textbf{X}$.
    \item Let $\textbf{X}$ be an adjunction space in which each $\phi_i$ is an open map. If each $X_i$ is first-countable, then so is the adjunction space $\textbf{X}$. 
    \item Let $\textbf{X}$ be an adjunction space in which the indexing set $I$ is countable. If every $X_i$ is Lindel\"of, then so is $\textbf{X}$.
    \item Let $\textbf{X}$ be an adjunction space in which the indexing set $I$ is countable. If each $X_i$ is separable, then so is $\textbf{X}$.
    \item Suppose each $X_i$ is second-countable. If $I$ is countable and each $\phi_i$ is an open map, then $\textbf{X}$ is also second-countable.
    \item If each $X_i$ is connected and each $A_{ij}$ is non-empty, then $\textbf{X}$ is connected.
    \item Let $\textbf{X}$ be an adjunction space in which every $A_{ij}$ is non-empty. If each $X_i$ is path-connected, then so is $\textbf{X}$.
    \item If $I$ is finite and each $X_i$ is compact, then so is $\textbf{X}$.
    \item Let $\textbf{X}$ be an adjunction space formed from a collection of $T_1$ spaces. If the canonical maps $\phi_i$ are all open, then $\textbf{X}$ is $T_1$.
    \item Let $\textbf{X}$ be an adjunction space in which each $X_i$ is locally-Euclidean. If each $\phi_i$ is an open embedding, then $\textbf{X}$ is also locally-Euclidean.
\end{enumerate}
\end{theorem}

\section{Non-Hausdorff Manifolds}
We will now use the adjunction formalism detailed in the previous section to construct non-Hausdorff manifolds. Observe that according to Lemma \ref{adjfAopenthen-phi-topoembeds} and items 5 and 10 of Theorem \ref{preservation of top prop} we already have the following.
    \begin{theorem}\label{Adj of manifolds gives generalised manifold}
    Let $\mathcal{F}$ be an adjunctive system consisting of countably-many (Hausdorff) manifolds $M_i$, in which each $A_{ij}$ is an open submanifold. Then the adjunction space subordinate to $\mathcal{F}$ is a locally-Euclidean second-countable space. 
    \end{theorem}
Let us denote by $\textbf{M}$ the adjunction space built according to the above. The open charts of $\textbf{M}$ at the point $[x,i]$ are given by $ (\phi_i(U), \varphi\circ \phi_i^{-1})$ where $(U, \varphi)$ is any open chart of $M_i$. Thus the manifold $\textbf{M}$ mirrors the local behaviour of the Hausdorff manifolds $M_i$. According to \ref{adjfAopenthen-phi-topoembeds} we may interpret $\textbf{M}$ as a locally-Euclidean, second-countable space that is covered by Hausdorff submanifolds. 

\subsection{Homeomorphic Boundaries} 
We will now identify conditions under which Hausdorff violation is guaranteed in the adjunction process. In order to do so, we will first recall a binary relation that formally encodes Hausdorff violation. We will opt for the notation used in \cite{hajicek1971causality, muller2013generalized}, though it should be noted that the same relation can also be found in \cite{kent2009note} and \cite[p. 67]{mardani2014topics} in essentially equivalent forms. \\

Let $\textbf{M}$ be a locally-Euclidean second-countable space, and consider the binary relation $\textsf{Y}$, defined as $ x \textsf{Y} y $ if and only if every pair of open neighbourhoods of $x$ and $y$ necessarily intersect. The relation $\textsf{Y}$ is reflexive and symmetric by construction, but it need not be transitive.\footnote{Interestingly, the relation $\textsf{Y}$ is weaker than the Hausdorff relation used in \cite{heller2011geometry}. It is likely that their relation is the transitive closure of $\textsf{Y}$, though a proof of this is beyond the scope of our paper.} It is well-known that a topological space is Hausdorff whenever convergent sequences have unique limits. Although the converse is not always true, in a first countable space one can always construct a sequence that converges to both elements of a Hausdorff-violating pair. As such, in our context the relation $x\textsf{Y}y$ asserts the existence of a sequence that converges to both $x$ and $y$ in $\textbf{M}$. Given a subset $V$ of $\textbf{M}$, we define $$ \textsf{Y}^V := \{ x\in \textbf{M} \ | \ \exists y( y \in V \wedge x \textsf{Y} y)\},     $$
thus $\textsf{Y}^V$ consists of all points in $\textbf{M}$ that are Hausdorff-inseparable from $V$. \\

In what follows, we will use the $\textsf{Y}$-relation to describe the Hausdorff violating points of the canonical subspaces $\phi_i(M_i)$. In order to do so, we will make use of the following definition.
\begin{definition}\label{DEF: homeomorphic boundaries}
Let $\mathcal{F}$ be an adjunction system as in Theorem \ref{Adj of manifolds gives generalised manifold}, and let $\textbf{M}$ be the corresponding adjunction space. We say that the gluing regions $A_{ij}$ have homeomorphic boundaries whenever each $A_{ij}$ is a proper subset of $M_i$ and each gluing map $f_{ij}$ can be extended to a homeomorphism $\tilde{f}_{ij}: Cl(A_{ij}) \rightarrow Cl(A_{ji})$. 
\end{definition}
This additional requirement will allow us to describe the Hausdorff-violating points using the $\textbf{M}$-relative boundaries of the canonical subspaces $M_i$. Before getting to the detailed description, we will first introduce some useful notation. \\

According to Theorem \ref{adjfAopenthen-phi-topoembeds}, we may view the subspaces $M_i$ as embedded inside the larger space $\textbf{M}$. To simplify notation, we will identify each $M_i$ with its image $\phi_i(M_i)$. Under this convention we may view $M_i$ as an honest Hausdorff (open) submanifold of $\textbf{M}$. In what follows, it will be useful to describe the boundaries of the $M_i$ in $\textbf{M}$. We will use the notation $\boldsymbol{\partial}$ to denote the $\textbf{M}$-relative boundaries of the $M_i$, that is, we define $$ \boldsymbol{\partial} M_i := \partial^{\textbf{M}}( \phi_i(M_i)).      $$ This boldface notation will similarly be used for the $\textbf{M}$-relative closures and interiors of sets. We will also identify each $A_{ij}$ in $M_i$ with its images $\phi_i(A_{ij})$ and $\phi_j(A_{ji})$ in $\textbf{M}$. In general, each $A_{ij}$ may have several $\textbf{M}$-relative boundary components. We denote these as follows: 
    $$ \boldsymbol{\partial}^i A_{ij} := \left(\partial^{\textbf{M}} \phi_i(A_{ij})\right) \cap \phi_i(M_i).          $$
We will now show that requiring an adjunction system to have homeomorphic gluing regions establishes a clear relationship between the boundary operator $\boldsymbol{\partial}$ and the binary relation $\textsf{Y}$.
\begin{lemma}\label{basic facts about homeomorphic bdrys}
Suppose that $\textbf{M}$ is a non-Hausdorff manifold built from an adjunction space in which the gluing regions have homeomorphic boundaries. Let $g_{ij}: \boldsymbol{\partial}^i A_{ij} \rightarrow\boldsymbol{\partial}^j A_{ij}$ be the map given by $g_{ij} := \phi_j \circ \tilde{f}_{ij} \circ \phi_i^{-1}$. 
    \begin{enumerate}
        \item The map $g_{ij}$ is a homeomorphism.
        \item For distinct points $[x,i]$ and $[y,j]$ in $\textbf{M}$, we have that $g_{ij}([x,i]) = [y,j]$ if and only if $[x,i] \textsf{Y} [y,j]$.
        \item $M_j \cap \boldsymbol{\partial}M_i = M_j \cap \textsf{Y}^{M_i}$. 
    \end{enumerate}
\end{lemma}
    \begin{proof}
    The first item is immediate from \ref{adjfAopenthen-phi-topoembeds} and \ref{DEF: homeomorphic boundaries}. Let $[x,i]$ and $[y,j]$ be distinct points in $\textbf{M}$. Suppose first that $g_{ij}([x,i]) = [y,j]$. By construction this means that $y = \tilde{f}_{ij}(x)$. In $\textbf{M}$, the gluing region $A_{ij}=A_{ji}$ equals the intersection $M_i \cap M_j$. Since the points $[x,i]$ and $[y,j]$ are distinct in $\textbf{M}$, it must be the case that $x$ and $y$ are in the boundaries of $A_{ij}$ and $A_{ji}$ respectively. Let $a_n$ be some sequence in $A_{ji}$ that converges to $y$ in $M_j$. Since $\tilde{f}_{ij}$ is a homeomorphism, it follows that $\tilde{f}^{-1}_{ij}(a_n)$ is a sequence in $A_{ij}$ that converges to $x$ in $M_i$. In the adjunction space $\textbf{M}$, the sequences $[a_n,j]$ and $[\tilde{f}^{-1}_{ij}(a_n), i]$ will be equal, and will converge to two distinct limits $[x,i]$ and $[y,j]$. Thus $[x,i] \textsf{Y} [y,j]$. The converse follows from part (1) and the fact that each $M_i$ is Hausdorff, thus has unique limits. 
    
    For the third item, suppose that $[y,j]$ is some element of $\boldsymbol{\partial}M_i$. Then there is a sequence in $M_i$ which converges to $[y,j]$ in $\textbf{M}$. Since $M_j$ is an open neighbourhood of $[y,j]$, without loss of generality we may assume that this sequence lies in $A_{ij}$. Thus $[y,j]$ lies in $\boldsymbol{\partial}^j A_{ij}$. It follows that $[y,j]\textsf{Y} g_{ij}([y,j])$, and thus $[y,j] \in \textsf{Y}^{M_i}$. The converse inclusion follows almost immediately from (1) and (2).  
    \end{proof}
The above result allows us to conclude that the Hausdorff-violating pairs of a non-Hausdorff manifold can be identified as the pairwise boundaries of open submanifolds that are glued together, provided that these boundaries exist and are homeomorphic to each other. The following result extends this idea.
\begin{lemma}\label{Mi are H-subman}
    Let $\textbf{M}$ be a non-Hausdorff manifold built from an adjunctive system $\mathcal{F}$. If the regions $A_{ij}$ have homeomorphic boundaries, then for each $M_i$ we have that $$ \textsf{Y}^{M_i} = \boldsymbol{\partial} M_i = \bigcup_{j\neq i} \boldsymbol{\partial}^j A_{ij}. $$
\end{lemma}
\begin{proof}
The inclusion $\textsf{Y}^{M_i} \subseteq \boldsymbol{\partial} M_i$ follows from Prop. \ref{basic facts about homeomorphic bdrys}.3. Suppose that $[y,j]$ is some element of $\boldsymbol{\partial} M_i$. Then there exists a sequence in $M_i$ that converges to $[y,j]$. Since $M_i$ and $M_j$ are open, without loss of generality we may assume that the sequence sits in the intersection $A_{ij}$. It follows that $[y,j]$ lies in $\boldsymbol{\partial}^j A_{ij}$. This shows that $\boldsymbol{\partial} M_i \subseteq \bigcup_{j\neq i} \boldsymbol{\partial}^j A_{ij}$. Suppose now that $[y,j]$ is in $\boldsymbol{\partial}^j A_{ij}$. We can use Prop \ref{basic facts about homeomorphic bdrys} to conclude that $[y,j] \textsf{Y} g_{ji}( [y,j])$. It follows that $\bigcup_{j \neq i} \boldsymbol{\partial}^j A_{ij} \subseteq \textsf{Y}^{M_i}$.
\end{proof}

\subsection{Paracompactness and Partitions of Unity}
It is well-known that Hausdorff manifolds are necessarily paracompact and admit partitions of unity subordinate to any open cover \cite{lee2013smooth}. In this section we will explore such results in the non-Hausdorff setting. \\

According to our discussion thus far, we may build non-Hausdorff manifolds from adjunction spaces. Since all Hausdorff manifolds are paracompact, we may expect there to be an analogue to the results of Theorem \ref{preservation of top prop} for paracompactness. Indeed this is true: if we consider finitely-built adjunction spaces and further restrict our attention to those adjunction spaces in which the gluing regions have homeomorphic boundaries, then we obtain the following result. 
\begin{theorem}\label{paracompactness preserved}
Let $\textbf{M}$ be a non-Hausdorff manifold built from an adjunction space in which the gluing regions have homeomorphic boundaries. If the indexing set $I$ of the adjunction system $\mathcal{F}$ is finite, then $\textbf{M}$ is paracompact.
\end{theorem}
\begin{proof}(Sketch)
Let $\mathcal{U}$ be some open cover of $\textbf{M}$. Consider the open covers $\mathcal{U}_i:= \{ \phi_i^{-1}(U) \ | \ U \in \mathcal{U} \}$ of the $M_i$. Since each $M_i$ is paracompact, each open cover $\mathcal{U}_i$ has a locally-finite refinement $\mathcal{V}_i$. We then define $$\mathcal{V} = \bigcup_{i \in I} \{\phi_i(V) \ | \ V \in \mathcal{V}_i \}.$$ Clearly $\mathcal{V}$ is a refinement of $\mathcal{U}$. To show that $\mathcal{V}$ is locally-finite around some point $[x,i]$, we will construct an open neighbourhood $W$ that intersects finitely-many members of $\mathcal{V}$. For any $j$ in $I$, there are three scenarios. 
\begin{enumerate}
    \item $[x,i]$ lies in $M_i \cap M_j$. In this case we may use the local finiteness of $\mathcal{V}_j$ in $M_j$ around $f_{ij}(x)$ to obtain some open neighbourhood $X_j$ that intersects finitely-many elements of $\mathcal{V}_j$. The set $\phi_j(X_j)$ will then intersect finitely-many elements of $\mathcal{V}$ coming from $\mathcal{V}_j$.  
    \item $[x,i]$ lies in $M_i \backslash \textbf{Cl}(M_j)$, that is, $[x,i]$ is in the set $\textbf{Int}(M_i \backslash M_j)$. This is an open set that is disjoint from $M_j$. 
    \item $[x,i]$ lies in $M_i \cap \boldsymbol{\partial} M_j$. In this case, we may use the local finiteness of $M_j$ around the boundary element $\tilde{f}_{ij}(x)$. This maps homeomorphically into some open neighbourhood of $x$ in $Cl^{M_i}(A_{ij})$. We may extend this to an open neighbourhood $X_j$ of $x$ in $M_i$. The set $\phi_i(X_j)$ then intersects finitely-many of the open sets in $\mathcal{V}$ coming from $\mathcal{V}_j$. 
\end{enumerate}
Intersecting all of the open sets mentioned above will yield an neighbourhood $W$ of $[x,i]$ that intersects at most finitely-many elements of $\mathcal{V}$. Note that $W$ is open since we have assumed that $I$ is finite. 
\end{proof}

The above result confirms that certain non-Hausdorff manifolds may be paracompact. In contrast to this, in the non-Hausdorff case there may be open covers that do not admit subordinate partitions of unity. This is immediate from standard results such as those in \cite{lee2013smooth}, however for the sake of completeness we will include a different argument. 
\begin{theorem}\label{no partitions of unity}
Let $\mathcal{U} = \{ U_\alpha\}_{\alpha \in A}$ be an open cover of a non-Hausdorff manifold $\textbf{M}$. If each $U_\alpha$ is Hausdorff, then the cover  $\mathcal{U}$ does not admit a partition of unity subordinate to it.   
\end{theorem}
\begin{proof}
Let $a$ and $b$ be elements of $\textbf{M}$ such that $a \textsf{Y}b$. Observe first that any continuous function $f: \textbf{M} \rightarrow X$ to a Hausdorff space $X$ will necessarily map $f(a) = f(b)$. Indeed, if it were the case that $f(a) \neq f(b)$, then we could apply the Hausdorff property to find two disjoint open sets $U$ and $V$ in $X$ separating $f(a)$ and $f(b)$. The continuity of $f$ would then cause the open sets $f^{-1}(U)$ and $f^{-1}(V)$ to contradict $a \textsf{Y} b$. \\
With this in mind, suppose towards a contradiction that there exists some partition of unity $\{ \psi_\alpha\}$ subordinate to the cover $\{ U_\alpha\}$. Consider the point $a$ in $\textbf{M}$. Since the partition of unity sums to $1$ at $a$, there must be at least one function $\psi_\alpha$ such that $\psi_\alpha(a) >0$. It follows that $$ a \in \{ x \in \textbf{M} \ | \ \psi_\alpha(x) \neq 0 \} \subseteq \overline{\{ x \in \textbf{M} \ | \ \psi_\alpha(x) \neq 0 \}} \subseteq U_\alpha. $$
However, by our observation will we have that $\psi_\alpha(b) = \psi_\alpha(a) > 0$. Therefore the above chain of inclusions applies equally to $b$, from which we may conclude that both $a$ and $b$ lie in the set $U_\alpha$. This contradicts our assumption that $U_\alpha$ is Hausdorff. 
\end{proof} 
 
It follows from the above that the open cover consisting of the canonical submanifolds $M_i$ cannot admit a subordinate partition of unity. It is likely that Theorem \ref{no partitions of unity} will result in some non-trivial obstructions when attempting to recreate the standard theory of differential geometry in the non-Hausdorff regime, though that is beyond the scope of this paper.
 
\subsection{$H$-submanifolds}
In the remainder of this section we will argue that all non-Hausdorff manifolds can be expressed as adjunction spaces. Before getting to the argument, we first need to review the notion of an $H$-submanifold. We recall the following definition, originally found in \cite{hajicek1971causality}.
\begin{definition}\label{Def: H-submanifold}
    Let $\textbf{M}$ be a non-Hausdorff manifold. A subset $V$ of $\textbf{M}$ is called an $H$-submanifold if $V$ is open, Hausdorff and connected, and is maximal with respect to these properties. 
\end{definition}
Since non-Hausdorff manifolds are locally-Euclidean, in particular they are locally-Hausdorff. We may use this observation, together with an appeal to Zorn's Lemma to argue the following.  
\begin{proposition}\label{h-subman forms open cover}
    The collection of $H$-submanifolds of a non-Hausdorff manifold forms an open cover. 
\end{proposition}

It is not immediately clear whether a given Hausdorff submanifold is an $H$-submanifold. Fortunately this has been resolved by Hajicek \cite[Thm. 2]{hajicek1971causality}, who provides a useful criterion for determining whether or not a given subspace is an $H$-submanifold. In our notation, Hajicek's criterion can be stated as follows.
\begin{theorem}[Hajicek's Criterion] \label{hajiceks criterion}
    A subset $V$ of a non-Hausdorff manifold $\textbf{M}$ is an $H$-submanifold if and only if the equality  $\boldsymbol{\partial}V= \textbf{Cl}(\textsf{Y}^V)$ holds.
\end{theorem}
We can use Hajicek's criterion together with Lemma \ref{Mi are H-subman} to conclude that any non-Hausdorff manifold $\textbf{M}$ built according to \ref{DEF: homeomorphic boundaries} admits the canonical subspaces $M_i$ as $H$-submanifolds. As we will see in Section 3, it is not always the case that the $M_i$ are the only $H$-submanifolds of $\textbf{M}$. 

\subsection{A Reconstruction Theorem}
We will now use Proposition \ref{h-subman forms open cover} to argue that all non-Hausdorff manifolds can be described via adjunction spaces. The idea behind this proof can be found in both \cite{OConnellthesis} where it is proved for vector bundles, and \cite{luc2020interpreting}, where it is proved for manifolds in general. For completeness we will provide a proof consistent with our notation.

\begin{theorem}\label{thm: reconstruction theorem}
If $\textbf{M}$ is a non-Hausdorff manifold then $\textbf{M}$ is homeomorphic to an adjunction space.
\end{theorem}
\begin{proof}
Consider the family $M_i$ of all $H$-submanifolds of $\textbf{M}$. We know from Prop. \ref{h-subman forms open cover} that the $M_i$ form an open cover of $\textbf{M}$. Since $\textbf{M}$ is second-countable, in particular $\textbf{M}$ is Lindel\"of. Thus without loss of generality we may assume that the $M_i$ form a countable open cover of $\textbf{M}$. We can then define an adjunction system $\mathcal{F}$ using:  \begin{itemize}
    \item the countable collection of $M_i$, with 
    \item $A_{ij} := M_i \cap M_j$, and
    \item $f_{ij}: A_{ij} \rightarrow M_j$ the identity map.
\end{itemize}
Clearly this collection satisfies the conditions of Definition \ref{adjsystemdef} and the criteria of Theorem \ref{Adj of manifolds gives generalised manifold}. Thus the associated adjunction space $\adj M_i$ is a well-defined locally-Euclidean second-countable space. We can invoke the universal property of adjunction spaces (Lemma \ref{adjuniversalprop}) to conclude that there exists a unique continuous map $g$ such that we have the following commutative diagram for every pair of $M_i$'s: 
\begin{center}
\begin{tikzcd}[row sep=3em]
A_{ij} \arrow[rr, hook] \arrow[dd, "id_{A_{ij}}"']                   &  & M_i \arrow[dd, "\phi_i", hook] \arrow[rddd, "\iota_i", bend left] &   \\
                                                                     &  &                                                                   &   \\
M_j \arrow[rr, "\phi_j"', hook] \arrow[rrrd, "\iota_j"', bend right] &  & \adj M_i \arrow[rd, "g" description, dotted]   &   \\
                                                                     &  &                                                                   & \textbf{M}
\end{tikzcd}
\end{center}
where the $\iota_{\cdot}$ are the inclusion maps, and
$ g: \adj M_i \rightarrow \textbf{M}$ is the map that acts by $[x,i]\mapsto x. $ 
This map is clearly open -- this follows from the fact that $\phi_i$ and $\iota_i$ are open maps for all $i$. Moreover, $f$ is a bijection -- the inverse is given by $f^{-1}(x) = [x,i]$ where $i$ is the index of any $M_i$ that contains $x$. By construction, this map is well-defined. Thus we have obtained a bijective, continuous, open map from $\adj M_i$ to $\textbf{M}$.
\end{proof}

\section{Examples}
As mentioned in the introduction, the prototypical example of a non-Hausdorff manifold is the so-called line with two origins. We will now introduce several more examples, built according to our adjunction space formalism. The following spaces will motivate the discussion in Section 4, where we will study the prospect of characterising the types of $H$-submanifolds that a non-Hausdorff manifold might admit.  

\subsection{The $n$-branched Real Line}
Take $I$ to be an indexing set of size $n$, and consider the adjunction system $\mathcal{F}$ where: 

\begin{itemize}
    \item each $M_i$ equal to a copy of the real line equipped with the standard topology, 
    \item each $A_{ij}$ equal to the set $(-\infty, 0)$ for all $i,j$ distinct, and
    \item each $f_{ij}:A_{ij} \rightarrow M_j$ is the identity map. 
\end{itemize}

Once we require that each $A_{ii}$ equals $M_i$, the data contained in $\mathcal{F}$ forms an adjunction system. The adjunction space subordinate to $\mathcal{F}$ will be a collection of $n$-many copies of the real line all glued to each other along the negative numbers, as pictured in Figure \ref{fig n-branched realline}. According to Theorems \ref{Adj of manifolds gives generalised manifold} and \ref{paracompactness preserved} the $n$-branched real line is a paracompact non-Hausdorff manifold.
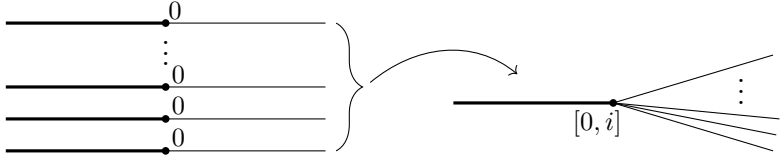
\begin{figure}
\begin{tikzpicture}[scale=0.85]
\draw[very thick] (0,0)--(2.5,0);
\draw[] (2.5,0)--(5,0);
\draw[very thick] (0,0.5)--(2.5,0.5);
\draw[] (2.5,0.5)--(5,0.5);
\draw[very thick] (0,1)--(2.5,1);
\draw[] (2.5,1)--(5,1);
\draw[very thick] (0,2)--(2.5,2);
\draw[] (2.5,2)--(5,2);


\fill(2.5,0) circle[radius=0.06cm] {};
\fill(2.5,0.5) circle[radius=0.06cm] {};
\fill(2.5,1) circle[radius=0.06cm] {};
\fill(2.5,2) circle[radius=0.06cm] {};

\node[] at (2.5, 1.65) {$\vdots$};
\node[] at (2.65,2.2) {\footnotesize{$0$}};
\node[] at (2.7,0.2) {\footnotesize{$0$}};
\node[] at (2.7,0.7) {\footnotesize{$0$}};
\node[] at (2.7,1.2) {\footnotesize{$0$}};

\draw [decorate,decoration={brace,amplitude=10pt,mirror,raise=4pt},yshift=0pt]
(5,0) -- (5,2) node [black,midway,xshift=0.8cm] {};

\draw[->] (5.7,1.07) to[out=40, in=140] (8,1.2);

\draw[very thick] (7,0.75)--(9.5,0.75);
\draw[] (9.5,0.75)--(12,0);
\draw[] (9.5,0.75)--(12.05,0.25);
\draw[] (9.5,0.75)--(12.1,0.5);
\draw[] (9.5,0.75)--(12,1.5);

\fill(9.5,0.75) circle[radius=0.06cm] {};

\node[] at (11.5, 1.05) {$\vdots$};
\node[] at (9.25,0.43) {\footnotesize{$[0,i]$}};

\end{tikzpicture} 
\caption{The construction of the $n$-branched real line}
\label{fig n-branched realline}
\end{figure}

Observe that the origins remain unidentified, thus there are $n$-many distinct equivalence classes $[0,i]$ in the $n$-branched real line. These points will violate the Hausdorff property. Moreover, there are $n$-many $H$-submanifolds equalling the canonically embedded $M_i$. \\

It should be noted that there are other interesting $1$-manifolds that can be built from finitely-many copies of the real line. Indeed, consider three copies of the real line, in which $A_{12} = (-\infty, 0)$, $A_{23}=(0,\infty)$, and $A_{13} = \emptyset$, and all $f$-maps equal the identity. This data defines an adjunction system, and the resulting space will be a non-Hausdorff manifold in which there are three distinct copies of the origin in which the $\textsf{Y}$-relation is not transitive.

\subsection{An Infinitely-Branching Real Line}
We will now construct a non-Hausdorff manifold using countably-many copies of the real line. Suppose that the indexing set $I$ coincides with the natural numbers, ordered linearly. Consider the triple $\mathcal{F}:=(\textbf{X},\textbf{A},\textbf{f})$, where: 
\begin{itemize}
    \item each $M_i$ equals $\mathbb{R}$ with the standard topology,
    \item $A_{ij} = \begin{cases}  (-\infty,i) & \textrm{if} \ i < j \\
    (-\infty,j) & \textrm{if} \ i > j \\
    \mathbb{R} & \textrm{if} \ i = j\end{cases}$, and 
    \item each $f_{ij}$ is the identity map on the appropriate domain.   
\end{itemize}
The reader may verify that this collection $\mathcal{F}$ does indeed form an adjunction system. The resulting adjunction space, which we denote by $\textbf{T}$, will be a countable collection of real lines, successively splitting in two at each natural number, as pictured in Figure \ref{Counterexample N}. 

\begin{figure}
\centering
    \begin{tikzpicture}[scale=0.65]

\draw[] (0,0)--(8,0);
\node[] at (-0.35,0) {$\cdots$};
\node[] at (8.45,0) {$\cdots$};

\draw[] (0,1)--(8,1);
\node[] at (-0.35,1) {$\cdots$};
\node[] at (8.45,1) {$\cdots$};
\fill[draw=black] (3,1) circle[radius=0.07] {};
\fill[draw=black] (3,0) circle[radius=0.07] {};

\draw[] (0,2)--(8,2);
\node[] at (-0.35,2) {$\cdots$};
\node[] at (8.45,2) {$\cdots$};
\fill[draw=black] (4,2) circle[radius=0.07] {};
\fill[draw=black] (4,1) circle[radius=0.07] {};

\draw[] (0,3)--(8,3);
\node[] at (-0.35,3) {$\cdots$};
\node[] at (8.45,3) {$\cdots$};
\fill[draw=black] (5,3) circle[radius=0.07] {};
\fill[draw=black] (5,2) circle[radius=0.07] {};

\draw[] (0,4)--(8,4);
\node[] at (-0.35,4) {$\cdots$};
\node[] at (8.45,4) {$\cdots$};
\fill[draw=black] (6,4) circle[radius=0.07] {};
\fill[draw=black] (6,3) circle[radius=0.07] {};

\draw[dotted] (0.25,0)--(0.25,4.25);
\draw[dotted] (0.5,0)--(0.5,4.25);
\draw[dotted] (0.75,0)--(0.75,4.25);
\draw[dotted] (1,0)--(1,4.25);
\draw[dotted] (1.25,0)--(1.25,4.25);
\draw[dotted] (1.5,0)--(1.5,4.25);
\draw[dotted] (1.75,0)--(1.75,4.25);
\draw[dotted] (2,0)--(2,4.25);
\draw[dotted] (2.25,0)--(2.25,4.25);
\draw[dotted] (2.5,0)--(2.5,4.25);
\draw[dotted] (2.75,0)--(2.75,4.25);
\draw[dotted] (3,1)--(3,4.25);
\draw[dotted] (3.25,1)--(3.25,4.25);
\draw[dotted] (3.5,1)--(3.5,4.25);
\draw[dotted] (3.75,1)--(3.75,4.25);
\draw[dotted] (4,2)--(4,4.25);
\draw[dotted] (4.25,2)--(4.25,4.25);
\draw[dotted] (4.5,2)--(4.5,4.25);
\draw[dotted] (4.75,2)--(4.75,4.25);
\draw[dotted] (5,3)--(5,4.25);
\draw[dotted] (5.25,3)--(5.25,4.25);
\draw[dotted] (5.5,3)--(5.5,4.25);
\draw[dotted] (5.75,3)--(5.75,4.25);
\draw[dotted] (6,4)--(6,4.25);
\node[] at (4, 5) {$\vdots$};

\draw [decorate,decoration={brace,amplitude=10pt,mirror,raise=4pt},yshift=0pt]
(8.75,-0.25) -- (8.75,5) node [black,midway,xshift=0.8cm] {};

\draw[->] (9.75,2.5) to[out=40, in=140] (13.5,2.5);

\draw[] (12,0.5)--(20,0.5);
\fill[draw=black] (14,0.5) circle[radius=0.07] {};
\draw[] (14,0.5)--(20,1.5);
\fill[draw=black] (15.5,0.75) circle[radius=0.07] {};
\draw[] (15.5,0.75)--(20,2.5);
\fill[draw=black] (17,1.325) circle[radius=0.07] {};
\draw[] (17,1.325)--(20,3.5);
\fill[draw=black] (18.5,2.415) circle[radius=0.07] {};
\draw (18.5,2.415)--(20,4.5); 
\node[] at (19, 4.5) {$\vdots$};

\end{tikzpicture} \\
    \caption{The construction of an infinitely-branching real line.}
    \label{Counterexample N}
\end{figure}
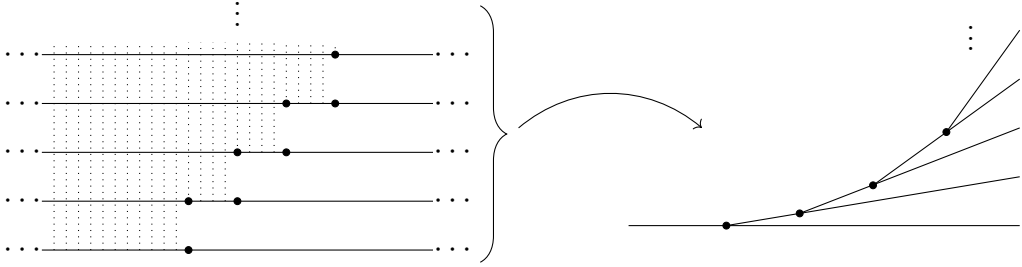

In this example, the copies of $\mathbb{R}$ naturally sit inside $\textbf{T}$ as $H$-submanifolds. In contrast to the previous example, there is an extra $H$-submanifold, given by $$ V = \bigcup_{i,j \in I} \phi_i(A_{ij}).$$ 
It should be noted that a similar space can be found as the rigid $1$-manifold of \cite{gartside2008homogeneous}.
\subsection{A $2$-Branched Euclidean Plane}
We can form branched planes by taking the product of the $n$-branched real line with a copy of $\mathbb{R}$. The resulting spaces will still be non-Hausdorff manifolds, and can be seen as a collection of Euclidean planes that branch out from each other along $x$-axes. \\

Consider the $2$-branched real plane. According to Theorem \ref{thm: reconstruction theorem} there should be a way to construct this space in terms of adjunctions. The obvious choice is to form an adjunction system using: 
\begin{enumerate}
    \item $M_1 = M_2 = \mathbb{R}^2$, 
    \item $A_{12} = \{ (x,y) \in \mathbb{R}^2 \ | \ y<0  \} $ is the open half-plane, and
    \item $f:A_{12} \rightarrow M_2$ is the identity map. 
\end{enumerate}
This collection satisfies the conditions of \ref{Adj of manifolds gives generalised manifold}, and the resulting adjunction space, denoted $\textbf{R}^2$, is a paracompact non-Hausdorff manifold. The construction is depicted in Figure \ref{fig: 2-branched plane}. The Hausdorff-violating points of $\textbf{R}^2$ lie on the two copies of the $x$-axis. For convenience, we will denote the two $x$-axes by $X_1$ and $X_2$.  

\begin{figure}
    \begin{tikzpicture}[scale=0.7]

\fill[color=gray!30] (2,-1)--(7,-2)--(7,-4.5)--(2,-3.5)--(2,-1);

\draw[] (2,1.5)--(7,0.5)--(7,-4.5)--(2,-3.5)--(2,1.5);
    \draw[dashed] (2,-1)--(7,-2);

\draw[dotted] (0,-2.5)--(2,-1);
\draw[dotted] (5,-3.5)--(7,-2);
\draw[dotted] (5,-4.75)--(7,-3.25);
\draw[dotted] (5,-6)--(7,-4.5);
\draw[dotted] (0,-5)--(2,-3.5);
\node[] at (2.75,-2) {$f(A)$};
\fill[color=white, opacity=0.8] (0,0)--(5,-1)--(5,-6)--(0,-5)--(0,0);
\fill[color=gray!30, opacity=0.8] (0,-2.5)--(5,-3.5)--(5,-6)--(0,-5)--(0,-2.5);
    \draw[] (0,0)--(5,-1)--(5,-6)--(0,-5)--(0,0);
        \draw[dashed] (0,-2.5)--(5,-3.5);   

\fill[gray!30] (11,-3)--(16,-4)--(16,-6.5)--(11,-5.5)--(11,-3);

\draw[] (11,-3)--(16,-4)--(16,-6.5)--(11,-5.5)--(11,-3);

\draw[] (11,-3)--(11.5,-0.3)--(16.5,-1.3)--(16,-4);
\fill[color=white, opacity=0.8] (11,-3)--(10,-0.8)--(15,-1.8)--(16,-4);
\draw[] (11,-3)--(10,-0.8)--(15,-1.8)--(16,-4);
\draw[very thick] (11,-3)--(16,-4);
 \fill[color=black] (13.5,-3.5) circle[radius=0.08] {};

\node[] at (-0.5,0.25) {$\mathbb{R}^2$};
\node[] at (0.5,-4) {$A$};
\node[] at (1.5,1.75) {$\mathbb{R}^2$};
\node[] at (16,-0.5) {$\textbf{R}^2$};

 \draw[->] (7.7,-3) to[out=20, in=160] (9.7,-3);

\end{tikzpicture} \\
\caption{An adjunction construction of a $2$-branched plane $\textbf{R}^2$. Here the thick line denotes two Hausdorff-inseparable copies of the $x$-axis.}
    \label{fig: 2-branched plane}
\end{figure}
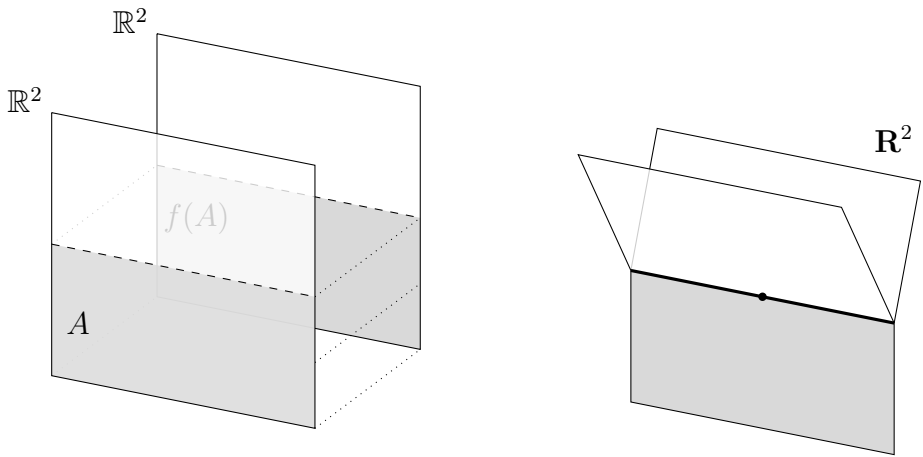

By Theorem \ref{Mi are H-subman} we see that the two copies of $\mathbb{R}^2$ that naturally sit inside $\textbf{R}^2$ will be $H$-submanifolds. However, the space $\textbf{R}^2$ admits many other $H$-submanifolds. We will now briefly describe one. Consider the subspace $V$ of $\textbf{R}$ defined by: $$ V = \textbf{R}^2 \backslash \big( \{ (x,0) \in X_1 \ | \ x \leq 0 \} \cup \{ (x,0) \in X_2 \ | \ x \geq 0 \} \big),      $$
that is, we remove from $\textbf{R}^2$ the non-positive $x$-axis from $X_1$ and the non-negative $x$-axis from $X_2$. Observe that both copies of the origin are removed. The subset $V$ is open in $\textbf{R}^2$ since its complement is closed. Moreover, the $\textsf{Y}$-set of $V$ is: 
$$ \textsf{Y}^V = \{ (x,0) \in X_1 \ | \ x > 0 \} \cup \{ (x,0) \in X_2 \ | \ x < 0 \}.     $$
By a simple analysis, one can see that $$  \boldsymbol{\partial} V = \textsf{Y}^{V} \cup \{ [(0,0), 1], [(0,0), 2] \} = \textbf{Cl}(\textsf{Y}^V),     $$
and thus Hajicek's criterion guarantees that $V$ is an $H$-submanifold of $\textbf{R}^2$. 

\subsection{A Non-Hausdorff Sphere}
We finish this section with an example of a compact non-Hausdorff manifold. Roughly speaking, this space will be a $2$-sphere in which the equatorial copy of $S^1$ is wrapped around itself so that the antipodal points become Hausdorff-inseparable, as pictured in Figure \ref{fig: twisted sphere alone}. We will denote this non-Hausdorff sphere by $\textbf{S}^2$.\\ 

\begin{figure}
    \centering
    \begin{tikzpicture}[tdplot_main_coords, scale = 2.5]

\shade[ball color = lightgray,
    opacity = 0.2
] (0,0,0) circle (1cm);

\tdplotsetrotatedcoords{0}{0}{0};
\draw[dashed,
    tdplot_rotated_coords,
    gray
] (-1,0,0.05) arc (180:360:1);

\tdplotsetrotatedcoords{0}{0}{0};
\draw[dashed,
    tdplot_rotated_coords,
    gray
] (-1,0,-0.05) arc (180:360:1);

\tdplotsetrotatedcoords{0}{0}{0};
\draw[dashed,
    tdplot_rotated_coords,
    gray
] (1,0,0.05) arc (0:160:1);

\tdplotsetrotatedcoords{0}{0}{0};
\draw[dashed,
    tdplot_rotated_coords,
    gray
] (1,0,-0.05) arc (0:160:1);

\draw[dashed, gray] (-0.95,0.3,0.05)--(-1,0,-0.05);
\draw[dashed, gray] (-1,0,0.05)--(-0.95,0.3,-0.05);

\end{tikzpicture}
    \caption{The non-Hausdorff sphere $\textbf{S}^2$, taken from \cite{hommelberg2014compact}.}
    \label{fig: twisted sphere alone}
\end{figure}

The original construction of $\textbf{S}^2$ uses a modified torus to obtain the doubled equator \cite{hommelberg2014compact}. A more succinct construct identifies antipodes of the $2$-sphere everywhere outside the equatorial copy of $S^1$ \cite{deeley2022fell}. Both of these constructions are not adjunction spaces, and it has even been suggested that such a colimit construction does not exist \cite{ruijter2017weak}. We will now remedy this by providing a construction of $\textbf{S}^2$ in terms of an adjunction of punctured $2$-spheres. We will start with the standard 2-sphere, embedded into $\mathbb{R}^3$ for convenience. Consider the equatorial copy of $S^1$ parameterised as $$  S^1 = \{ (x,y,0) \ | \ x^2 + y^2 = 0 \}.$$
Let $a = (1,0,0)$ and $b = (-1,0,0)$ be antipodes on $S^1$. Consider the following punctured spheres: 
\begin{itemize}
    \item $M_1$ and $M_2$ are  $S^2 \backslash \{ a \}$, and
    \item $M_3$ and $M_4$ are  $S^2 \backslash \{ b \}$.
\end{itemize}

We will now glue the $M_i$ in such a manner that halves of each $S^1$ eventually form a double-covered equator. Within each $M_i$ there are two arcs connecting the antipodal points $a$ and $b$. We denote these by $L^i_+$ and $L^i_-$, that is, $$  L^i_+ := \{ (x,y,0) \in M_i \ | \ y\geq 0 \} \ \textrm{and} \ L^i_- := \{ (x,y,0) \in M_i \ | \ y\leq 0 \}.  $$
Observe that in each $M_i$ the segments $L^i_{\pm}$ are closed subsets. The following table contains all of the pairwise gluing regions $A_{ij}$. 

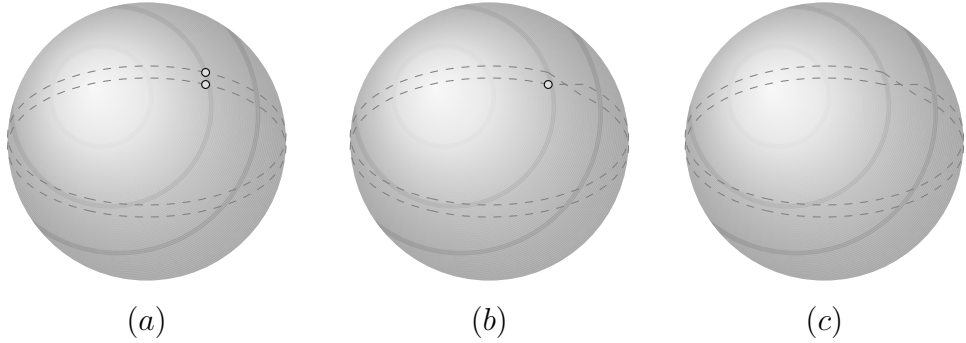
\begin{figure}
    \begin{center}
\begin{tikzpicture}[tdplot_main_coords, scale = 1.85]

\shade[ball color = lightgray,
    opacity = 0.2
] (0,0,0) circle (1cm);

\tdplotsetrotatedcoords{0}{0}{0};
\draw[dashed,
    tdplot_rotated_coords,
    gray
] (-1,0,0.05) arc (180:360:1);

\tdplotsetrotatedcoords{0}{0}{0};
\draw[dashed,
    tdplot_rotated_coords,
    gray
] (-1,0,-0.05) arc (180:360:1);

\tdplotsetrotatedcoords{0}{0}{0};
\draw[dashed,
    tdplot_rotated_coords,
    gray
] (1,0,0.05) arc (0:180:1);

\tdplotsetrotatedcoords{0}{0}{0};
\draw[dashed,
    tdplot_rotated_coords,
    gray
] (1,0,-0.05) arc (0:180:1);

\draw[fill = lightgray!50] (-1,0,0.05) circle (0.8pt);
\draw[fill = lightgray!50] (-1,0,-0.05) circle (0.8pt);

\node[] at (0,0,-1.5) {$(a)$};
 
\end{tikzpicture}
\hspace{0.05cm}
\begin{tikzpicture}[tdplot_main_coords, scale = 1.85]

\shade[ball color = lightgray,
    opacity = 0.2
] (0,0,0) circle (1cm);

\tdplotsetrotatedcoords{0}{0}{0};
\draw[dashed,
    tdplot_rotated_coords,
    gray
] (-1,0,0.05) arc (180:360:1);

\tdplotsetrotatedcoords{0}{0}{0};
\draw[dashed,
    tdplot_rotated_coords,
    gray
] (-1,0,-0.05) arc (180:360:1);

\tdplotsetrotatedcoords{0}{0}{0};
\draw[dashed,
    tdplot_rotated_coords,
    gray
] (1,0,0.05) arc (0:160:1);

\tdplotsetrotatedcoords{0}{0}{0};
\draw[dashed,
    tdplot_rotated_coords,
    gray
] (1,0,-0.05) arc (0:160:1);

\draw[dashed, gray] (-0.95,0.3,0.05)--(-1,0,-0.05);
\draw[dashed, gray] (-1,0,0.05)--(-0.95,0.3,-0.05);

\draw[fill = lightgray!50] (-1,0,-0.05) circle (0.8pt);
 
\node[] at (0,0,-1.5) {$(b)$}; 
 
\end{tikzpicture}
\hspace{0.05cm}
\begin{tikzpicture}[tdplot_main_coords, scale = 1.85]

\shade[ball color = lightgray,
    opacity = 0.2
] (0,0,0) circle (1cm);

\tdplotsetrotatedcoords{0}{0}{0};
\draw[dashed,
    tdplot_rotated_coords,
    gray
] (-1,0,0.05) arc (180:360:1);

\tdplotsetrotatedcoords{0}{0}{0};
\draw[dashed,
    tdplot_rotated_coords,
    gray
] (-1,0,-0.05) arc (180:360:1);

\tdplotsetrotatedcoords{0}{0}{0};
\draw[dashed,
    tdplot_rotated_coords,
    gray
] (1,0,0.05) arc (0:160:1);

\tdplotsetrotatedcoords{0}{0}{0};
\draw[dashed,
    tdplot_rotated_coords,
    gray
] (1,0,-0.05) arc (0:160:1);

\draw[dashed, gray] (-0.95,0.3,0.05)--(-1,0,-0.05);
\draw[dashed, gray] (-1,0,0.05)--(-0.95,0.3,-0.05);
 
\node[] at (0,0,-1.5) {$(c)$};
 
\end{tikzpicture}
\end{center}
    \caption{A successive construction of $\textbf{S}^2$. Here $(a)$ is the adjunction of the spaces $M_1$ and $M_2$. Figure (b) represents the gluing of $M_1, M_2$ and $M_3$. Figure (c) is the final space.}
    \label{FIG: Constructing the twisted sphere}
\end{figure}
\begin{center}
\begin{tabular}{ c|c|c|c|c} 
 \ & $M_1$ & $M_2$ & $M_3$ & $M_4$ \\
 \hline
 $M_1$ & $M_1$ & $M_1 \backslash S^1$ & $M^1 \backslash L^1_-$ & $M^1 \backslash L^1_+$ \\
 $M_2$ & $M^2 \backslash S^1$ & $M_2$ & $M^2 \backslash L^2_+$ & $M^2 \backslash L^2_-$ \\
 $M_3$ & $M^3 \backslash L^3_-$ & $M^3 \backslash L^3_+$ & $M_3$ & $M^3 \backslash S^1$ \\
 $M_4$ & $M^4 \backslash L^4_+$ & $M^4 \backslash L^4_-$ & $M^4 \backslash S^1$ & $M_4$ \\
\end{tabular}
\end{center}
The above data indeed defines an adjunction system. It is perhaps easiest to visualise the gluing as a successive process. Figure \ref{FIG: Constructing the twisted sphere} depicts the stages in the construction.

\section{The Characterisation of $H$-Submanifolds}

In Section 3.3 we saw the $2$-branched Euclidean plane, a non-Hausdorff manifold that admits infinitely-many $H$-submanifolds. Of these infinite $H$-submanifolds there are two that stand out -- the images of the two copies of the plane. This raises an interesting question: given an adjunction space that forms a non-Hausdorff manifold, what topological properties characterise the images $\phi_i(M_i)$? In this section we will take steps towards this characterisation. 

\subsection{Simple Non-Hausdorff Manifolds}

We will restrict our attention to a basic type of non-Hausdorff manifold. Following the terminology of \cite{muller2013generalized}, we will refer to these as \textit{simple} non-Hausdorff manifolds. The definition is as follows. 

\begin{definition}\label{DEF: Simple NH Manifolds}
A non-Hausdorff manifold $\textbf{M}$ is called \textit{simple} if it is homeomorphic to an adjunction space $\adj M_i$ in which:
\begin{itemize}
    \item each $M_i$ is the same Hausdorff manifold $M$,
    \item the indexing set $I$ is finite,
    \item each $A_{ij}$ is the same connected open subspace $A$ that has connected boundary, and 
    \item each gluing map $f_{ij}$ is the identity. 
\end{itemize}
Similarly, an adjunction system $\mathcal{F}$ is called simple if its adjunction space yields a simple non-Hausdorff manifold.
\end{definition}

The above definition captures the idea of a non-Hausdorff manifold branching out finitely-many times from a single submanifold. Note that the $n$-branched real line and the $2$-branched plane of Section 3 are both simple, whereas the infinitely-branching real line is not. \\

The main benefit of considering simple non-Hausdorff manifolds is that they admit a particularly easy description of Hausdorff violation. Observe that the component spaces of a simple non-Hausdorff manifold will have homeomorphic boundaries, since we may use the identity map to define each extension $\tilde{f}_{ij}$, and thus the boundary-maps $g_{ij}$ equal $\phi_j \circ \phi_i^{-1}$ (cf. \ref{DEF: homeomorphic boundaries}). We also make the following useful observation, which follows immediately from Definition \ref{DEF: Simple NH Manifolds}.
\begin{lemma}\label{simple then trivial bdry in M}
If $\textbf{M}$ is simple, then $\boldsymbol{\partial}^iA_{ij} = \boldsymbol{\partial}^i A =  \boldsymbol{\partial}^i A_{ik}$ for all $i,j,k$ in $I$. 
\end{lemma}
It follows from the above that the Hausdorff-violating sets in a simple non-Hausdorff manifold will consist of the $I$-many (disjoint) connected components of the boundary of $A$. Thus we may conclude that the non-Hausdorff sphere of Section 3.4 is not simple, since the set $\textsf{Y}^{\textbf{S}^2}$ is path-connected.

\subsection{A Generalisation of M\"uller's Theorem}
We will now argue that simple non-Hausdorff manifolds admit a characterisation of their canonical subspaces $M_i$. There is an argument due to M\"uller \cite{muller2013generalized} that certain simple branched Minkowski spaces admit a characterisation in terms of existent limit points. We will extend this result by passing into a broader generality, and by using a weaker condition.\\

Recall that Hajicek's criterion demands that each $H$-submanifold satisfies $\boldsymbol{\partial} V = \textbf{Cl}(\textsf{Y}^V)$. We will argue that in a simple non-Hausdorff manifold, the canonical subspaces $M_i$ are the unique $H$-submanifolds satisfying the stricter equality $\boldsymbol{\partial} V = \textsf{Y}^V$. Our argument will be via induction on the size of the indexing set underlying $\mathcal{F}$. In order to do so, we will make use of the following key lemma. 
\begin{lemma}\label{simple then y-set disjoint or full}
Suppose that $\textbf{M}$ is simple and $V$ is an $H$-submanifold of $\textbf{M}$ satisfying $\boldsymbol{\partial} V= \textsf{Y}^V$. Then for any boundary component $\boldsymbol{\partial}^i A$, either $\textsf{Y}^V $ and $\boldsymbol{\partial}^i A$ are disjoint, or $\boldsymbol{\partial}^i A$ is contained within $\textsf{Y}^V$.
\end{lemma}
\begin{proof}
By \ref{DEF: Simple NH Manifolds} $\boldsymbol{\partial}^i A$ is connected, so it suffices to show that the intersection $\textsf{Y}^V \cap \boldsymbol{\partial}^i A$ is clopen in $\boldsymbol{\partial}^i A$. Observe first that by assumption $\textsf{Y}^V$ coincides with the boundary $\boldsymbol{\partial} V$, so it is closed in $\textbf{M}$. Moreover, observe that Lemma \ref{simple then trivial bdry in M} yields the equality $$  \textsf{Y}^V \cap \boldsymbol{\partial}^i A = \bigcup_{k\neq i} \phi_i \circ \phi_k^{-1}\left(V \cap \boldsymbol{\partial}^k A\right).$$ Since the sets $V \cap \boldsymbol{\partial}^k A$ are all open and they map homeomorphically into $\boldsymbol{\partial}^i A = \boldsymbol{\partial}^i A$, we may conclude that $\textsf{Y}^V \cap \boldsymbol{\partial}^i A$ is a union of open sets.  
\end{proof}

We may use the above lemma together with Theorem \ref{Mi are H-subman} to obtain the following result. 
\begin{theorem}\label{main thm binary case}
Let $\textbf{M}$ be a simple non-Hausdorff manifold that is homeomorphic to a binary adjunction space $M_1 \cup_f M_2$. If $V$ is an $H$-submanifold of $\textbf{M}$ that satisfies the equality $\boldsymbol{\partial}V = \textsf{Y}^V$, then either $V=M_1$ or $V=M_2$.
\end{theorem}
\begin{proof}
By definition the gluing regions in a simple non-Hausdorff manifolds have homeomorphic boundaries, so we may apply Lemma \ref{Mi are H-subman} to conclude that both $M_i$ satisfy $\boldsymbol{\partial} M_i = \textsf{Y}^{M_i}$. Suppose towards a contradiction that there is some $H$-submanifold $V$ of $\textbf{M}$ that satisfies $\boldsymbol{\partial} V = \textsf{Y}^V$, but is distinct from both $M_1$ and $M_2$. Then there are points $[x,1]$ in $V \backslash M_2$ and $[y,2]$ in $V\backslash M_1$. Let $\gamma$ be a path connecting $[x,1]$ to $[y,2]$ in $V$. Consider some element $z$ in the $\gamma$-relative boundary $\partial^\gamma(M_1)$ that is distinct from $[x,1]$. Then $z$ lies in $\boldsymbol{\partial} M_1$ as well. Since $\boldsymbol{\partial} M_1 = \textsf{Y}^{M_1}$, there exists some $w$ in $M_2$ such that $z \textsf{Y} w$. Since $z$ lies on the curve $\gamma$ (which is a curve in $V$), it follows that the element $w$ is a member of $\textsf{Y}^V$. Therefore $\textsf{Y}^V \cap \boldsymbol{\partial}^2 A \neq \emptyset$, so $\boldsymbol{\partial}^2 A \subseteq \textsf{Y}^V$ by Lemma \ref{simple then trivial bdry in M}. We may now repeat the same argument, this time starting with the subset $M_2 \cap \gamma$. We will obtain an element $w'$ that lies in the set $\textsf{Y}^V \cap \boldsymbol{\partial}^1 A$. This implies that both $\boldsymbol{\partial}^1 A$ and $\boldsymbol{\partial}^2 A$ are subsets of $\textsf{Y}^V$, which contradicts $V$ as Hausdorff.  
\end{proof}

In order to continue an inductive argument, we need to appeal to the adjunctive subspaces of Section 1.2. Observe first that any subsystem $\mathcal{G}$ of a simple adjunction system $\mathcal{F}$ will again be simple. We will also have the following collection of facts, which follow trivially from \ref{thm adjunctive subspaces}.

\begin{lemma}\label{V restricts to Hsub of subadj space}
Let $\mathcal{F}$ be a simple adjunction system, and $\mathcal{G} \subset \mathcal{F}$. Denote by $g: \bigcup_{\mathcal{G}} M_j \rightarrow \adj M_i$ the map sending each $\llbracket x,i\rrbracket$ to $[x,i]$. Let $V$ be an $H$-submanifold of $\adj M_i$. Then \begin{enumerate}
    \item if $V \subseteq g\left( \bigcup_{\mathcal{G}} M_j\right)$, then $\llbracket V \rrbracket := g^{-1}(V)$ is an $H$-submanifold of $\bigcup_{\mathcal{G}} M_j$, and 
    \item if additionally $V$ satisfies $\boldsymbol{\partial} V = \textsf{Y}^V$, then so does $\llbracket V \rrbracket$.
\end{enumerate}
\end{lemma}

We will now use the above results to prove the main theorem of this paper. 

\begin{theorem}\label{main thm}
Let $\textbf{M}\cong \adj M_i$ be a simple non-Hausdorff manifold. If $V$ is an $H$-submanifold of $\textbf{M}$ that satisfies the equality $\boldsymbol{\partial}V = \textsf{Y}^V$, then $V=M_i$ for some $i$ in $I$.
\end{theorem}
\begin{proof}
We proceed via induction on the size of the adjunction space $\mathcal{F}$, that is, by induction on the size of the indexing set $I$. If $n=2$, then the result follows from the argument detailed in Theorem \ref{main thm binary case}. Suppose that the hypothesis holds for adjunction spaces of size $n$, and let $\textbf{M}$ be an adjunction space of size $n+1$. Suppose towards a contradiction that there exists some $H$-submanifold $V$ that is distinct from $M_1, ..., M_{n+1}$ yet satisfies the condition $\boldsymbol{\partial} V = \textsf{Y}^V$. \\

Observe first that $V$ cannot be contained within some finite union of $n$-many $M_i$'s. Indeed -- if it were then we could apply Lemma \ref{V restricts to Hsub of subadj space} and restrict $V$ to $\llbracket V \rrbracket$. We may then apply the induction hypothesis to conclude that $\llbracket V \rrbracket$ equalled some $M_i$ in the adjunctive subspace, and thus $V$ would equal that same $M_i$ in $\textbf{M}$. It must therefore be the case that $V$ is not a subset of some union of $n$-many $M_i$'s. Moreover, by maximality it cannot be the case that $V \subseteq M_{n+1}$. Therefore, we may infer that: 
\begin{enumerate}
    \item there exists some element $x$ in $V$ that lies in the difference $M_{n+1} \backslash \bigcup_{i=1}^n M_i$, and 
    \item there exists some element $y$ in $V$ that lies in the difference $\textbf{M}  \backslash M_{n+1}$.\footnote{Throughout this proof we will suppress the equivalence classes of points in $\textbf{M}$ for readability.}
\end{enumerate}
We may now proceed in a manner similar to that of Theorem \ref{main thm binary case}, with some key modifications for this more-complicated scenario. Let $\gamma$ be a path in $V$ that connects $x$ to $y$ in $V$. We will now use $\gamma$ to yield a contradiction. 
\begin{enumerate}
    \item Consider the set $\gamma \cap M_{n+1}$. The path $\gamma$ might enter and exit $M_{n+1}$ several times. So, consider the connected component $C$ of $\gamma \cap M_{n+1}$ that contains the endpoint $x$. The existence of the endpoints $x$ and $y$ ensure that the set $C$ is an open subset of $\gamma$ that is both non-empty and proper. As such, $C$ has a $\gamma$-relative boundary. Let $z$ be the element of $\partial^\gamma C$ distinct from $x$. Then $z$ is also an element of $\boldsymbol{\partial} M_{n+1}$. By Lemma \ref{basic facts about homeomorphic bdrys}, there must exist some $z'$ in $M_{n+1}$ such that $z \textsf{Y} z'$. Since $z$ lies in $\gamma$, which lies in $V$, we may conclude that $z' \in \textsf{Y}^V$. 
    Thus $z'$ lies in both $\boldsymbol{\partial}^{n+1} A$ and $\textsf{Y}^V$. According to Lemma \ref{simple then y-set disjoint or full}, we may conclude that $\boldsymbol{\partial}^{n+1}A$ is a subset of $\textsf{Y}^V$. 
    \item Consider the set $\gamma \backslash A$. This is a non-empty, closed proper subset of $\gamma$. Let $D$ be the connected component of $\gamma \backslash A$ that contains $x$. The set $D$ has two elements in its $\gamma$-relative boundary -- the first is $x$, and the second will be some other element $w$ in $\gamma$. Then $w \in \partial^\gamma A$, from which it follows that $w$ also lies in the set $\boldsymbol{\partial}^{n+1} A$. Since $w$ is an element of $\gamma$, which is a path in $V$, it follows that $V \cap \boldsymbol{\partial}^{n+1}A$ contains $w$.  
\end{enumerate}
We have thus arrived at our contradiction -- according to items (1) and (2) above the intersection $V \cap \textsf{Y}^V$ is non-empty, which contradicts our assumption that $V$ is Hausdorff. We may therefore conclude that there is no $V$ in $\textbf{M}$ that satisfies $\boldsymbol{\partial} V = \textsf{Y}^V$ and is distinct from each $M_1,...,M_{n+1}$. 
\end{proof}

\section{Conclusion}
In this paper we have introduced a general theory for simultaneously gluing arbitrarily-many topological spaces together along open subsets. In Section 2 we saw that if we consider a countable collection of Hausdorff manifolds that are glued along homeomorphic open subspaces, then we will always obtain a locally-Euclidean, second-countable space. Theorem \ref{thm: reconstruction theorem} confirmed that all non-Hausdorff manifolds can be described in this manner. According to Theorems \ref{preservation of top prop} and \ref{paracompactness preserved} we saw that such spaces are $T_1$, and may also be paracompact, provided that certain criteria is met. However, due to \ref{no partitions of unity} they may have open covers which do not admit partitions of unity subordinate to them. Upon requiring that the gluing regions have pairwise homeomorphic boundaries, we saw in Theorems \ref{basic facts about homeomorphic bdrys} and \ref{Mi are H-subman} that we can describe the Hausdorff-violating points of a non-Hausdorff manifold via these boundary components. This description can be summarised by the equality: $\boldsymbol{\partial}M_i = \textsf{Y}^{M_i}$. \\

Finally, we extended and improved a result of M\"uller \cite{muller2013generalized} to include general simple non-Hausdorff manifolds. Theorem \ref{main thm} shows that any simple non-Hausdorff manifold $\textbf{M}$ admits $n$-many $H$-submanifolds that satisfy the reduced Hajicek criterion $\boldsymbol{\partial} V = \textsf{Y}^V$. These special $H$-submanifolds are given uniquely by the $M_i$, which were the spaces used to construct $\textbf{M}$ in the first place. \\

It is not clear how to further generalise the result of \ref{main thm}. The requirement of simplicity was used in two essential ways: first, we restricted our attention to finite-sized adjunction systems in order to be able to perform an inductive argument, and we demanded the heavy restriction of \ref{simple then trivial bdry in M} in order to use Lemma \ref{simple then y-set disjoint or full}, which was an integral part of our eventual argument. It may be the case that the reduced Hajicek criterion is not an appropriate condition to uniquely characterise canonical subspaces of a non-simple non-Hausdorff manifold. Indeed -- the infinitely-branched real line of Section 3.2 has an ``extra" $H$-submanifold that still satisfies the property $\boldsymbol{\partial}V = \textsf{Y}^V$. A more nuanced condition may be required for the general argument. 

\subsection*{Acknowledgements}
This paper was made possible by the funding received by the Okinawa Institute of Science and Technology. I'd like to thank Yasha Neiman, Slava Lysov and Tim Henke for the productive discussions, and KP Hart and Andrew Lobb for the advice. 
\printbibliography

\end{document}